\documentclass[11pt,reqno,a4paper]{amsart}

\usepackage{xcolor}
\usepackage{ulem}
\usepackage{amssymb,epsfig,graphics}
\usepackage{tikz}

\newtheorem{theorem}{Theorem}[section]
\newtheorem{proposition}[theorem]{Proposition}

\newtheorem{lemma}[theorem]{Lemma}
\newtheorem{sub-lemma}[theorem]{Sub-Lemma}

\newtheorem{example}[theorem]{Example}

\def\A{\mathcal{A}}
\def\L{\mathcal{L}}
\def\C{\mathcal{C}}

\def\H{\mathcal{H}}
\def\I{\mathcal{I}}
\def\M{\mathcal{M}}
\def\N{\mathcal{N}}
\def\R{\mathcal{R}}
\def\V{\mathcal{V}}

\def\EE{\mathbb{E}}
\def\NN{\mathbb{N}}
\def\PP{\mathbb{P}}
\def\RR{\mathbb{R}}

\DeclareMathOperator{\diam}{diam}

\def\M{\mathcal{M}}

\let\eps=\varepsilon

\def\B{\mathcal{B}}

\def\RR{{\mathbb R}}
\def\1{{{\mathit 1} \!\!\>\!\! I} }

\begin{document}

\title[Applications of spatio-temporal rare events processes]{Application of the convergence of the spatio-temporal processes for visits to small sets}
\author{Fran\c{c}oise P\`ene \and Beno\^\i t Saussol}
\address{Univ Brest, Universit\'e de Brest, LMBA, Laboratoire de
Math\'ematiques de Bretagne Atlantique, CNRS UMR 6205, Brest, France}
\email{francoise.pene@univ-brest.fr}
\email{benoit.saussol@univ-brest.fr}
\keywords{}
\subjclass[2000]{Primary: 37B20}
\begin{abstract}
The goal of this article is to point out the importance of spatio-temporal processes in different questions of quantitative recurrence.
We focus on applications to the study of the number of visits to a small set before the first visit to another set (question arising from a previous work by Kifer and Rapaport), the study of high records, the study of line processes, the study of the time spent by a flow in a small set. We illustrate these applications by results on billiards or geodesic flows. This paper contains in particular new result of convergence in distribution of the spatio temporal processes associated to visits by the Sinai billiard flow to a small neighbourhood of orbitrary points in the billiard domain.
\end{abstract}
\date{\today}
\maketitle
\bibliographystyle{plain}
\tableofcontents

\section{Introduction}

Let $(\Omega,\mathcal F,\mu,T)$ or $(\Omega,\mathcal F,\mu,Y=(Y_t)_{t\ge 0})$ be a probability preserving dynamical system in discrete or continuous times.
Let $(A_\varepsilon)_\varepsilon$ be a family of measurable subsets of $\Omega$
with $\mu(A_\varepsilon)\rightarrow 0+$ as $\varepsilon\rightarrow 0$.
Given a family of measurable normalization functions $H_\varepsilon:A_\varepsilon\rightarrow V$
where $V$ is a locally compact metric space endowed with its Borel $\sigma$-algebra $\mathcal V$, we study the family of spatio-temporal point processes $(\mathcal N_\varepsilon)_\varepsilon$ on $[0,+\infty)\times V$ given by
\begin{equation}\label{PointProcessT}
 \mathcal N_\varepsilon(x):=
\mathcal N(T,A_\varepsilon,h_\varepsilon,H_\varepsilon):=\!\!\!\!\!
\sum_{n\ge 1\ :\ T^n(x)\in A_\varepsilon} \delta_{(nh_\varepsilon,H_\varepsilon(T^n(x)))} \quad\mbox{for a map }T
\end{equation}
or
\begin{equation}\label{PointProcessflot}
 \mathcal N_\varepsilon(x):=\mathcal N(Y,A_\varepsilon,h_\varepsilon,H_\varepsilon)=\!\!\!\!\!\!
\sum_{t>0\ :\ Y_t\mbox{ enters } A_\varepsilon} \delta_{(th_\varepsilon,H_\varepsilon(Y_t(x)))} \quad\mbox{for a flow }Y\, .
\end{equation}
We are interested in results of convergence in distribution of $(\mathcal N_\eps)_{\eps>0}$ to a point process $\mathcal P$ as $\varepsilon\rightarrow 0$ with a particular focus on applications of results of such kind.
Various results of convergence of such processes to Poisson point processes have been proved in \cite{ps20,jpz} for billiard maps and flows.\\
Let us point out the fact that these spatio-temporal processes contain much information: they do not only contain information on the visit time but they also contain informations on the 
spatial position at these visit times. For these reasons, on  may extract further information from results of convergence of these processes. 
Among the applications that have already been studied, let us mention:
\begin{itemize}
\item Study of the visits in a small neighborhood of an hyperbolic periodic point
of a transformation (see \cite[Section 5]{ps20}, with application to Anosov maps).\\
Such visits occurs by clusters (once a point visits such a neighbourhood, it stays close to the periodic point during an unbounded time before living this area). The idea we used to study these clusters was to consider a process $\mathcal N_\varepsilon$ corresponding to the last (or first) position of the clusters.
\item Convergence of a normalized Birkhoff sum processes
\[
\left(\left(n^{-\frac 1\alpha}\sum_{k=0}^{\lfloor nt\rfloor-1}f\circ T^k\right)_{t\ge 0}\right)_{n\ge 1}
\]
to an $\alpha$-stable process.
In \cite{marta} Tyran-Kami\'{n}ska provided criteria ensuring such a result. One of the conditions is the convergence of 
\[
\mathcal N_{1/n}=\mathcal N(T,\{|f|>\gamma n^{\frac 1\alpha}\},1/n,n^{-\frac 1\alpha}f(\cdot))
\]
(for every $\gamma>0$) to some Poisson point process. The general results of
\cite{ps20} combined with the criteria of \cite{marta} have been used in \cite{jpz}
to prove convergence to a L\'evy process for the Birkhoff sum process of H\"older
observable of billiards in dispersing domains with cusps.
\end{itemize}
We won't detail again the above applications. Our goal here is to emphasize on further ones.

After recalling  in Section~\ref{sec:gene} below the general results of convergence of spatio-temporal point processes to Poisson point processes established in \cite{ps20}, we present
in the remaining sections four other important applications of such convergence results:
\begin{itemize}
\item The number of visits  to (or of the time spent in) a small set before the first visit to a second small set (motivated by Kifer and Rapaport~\cite{kr}), with application to the Sinai billiard flow  with finite horizon,
\item The evolution of the number of records larger than some threshold, with an application to billiards with corners and cusps of order larger than 2,
\item The Line process of random geodesics (motivated by Athreya, Lalley, Sapir and Wroten \cite{line}),
\item The time spent by a flow in a small set, with application to the Sinai billiard flow with finite horizon.
\end{itemize}
Appendix \ref{appendbill} contains a new theorem of convergence of point processes for the Sinai billiard flow and for neighborhoods of arbitrary positions in the billiard domain, which is used in the examples that illustrate the applications above. Finally we also present an application to the closest approach by the billiard flow.

\section{Convergence results for transformations and special flows}\label{sec:gene}
We set $E:=[0,+\infty)\times V$ and we endow it with its Borel
$\sigma$-algebra $\mathcal E=\mathcal B([0,+\infty))\otimes \mathcal V$.
We also consider the family of measures $(m_\eps)_\eps$ on $(V,\mathcal V)$ defined by
\begin{equation}
m_\eps:=\mu(H_\eps^{-1}(\cdot)|A_\eps) 
\end{equation}
and 
$\mathcal W$ a family stable by finite unions and intersections of relatively compact open 
subsets of $V$, that generates the $\sigma$-algebra $\mathcal V$.
Let $\lambda$ be the Lebesgue measure on $[0,\infty)$.

We will approximate the point process defined by \eqref{PointProcessT}
or \eqref{PointProcessflot} by a Poisson point process
on $ E$.
Given $\eta$ a $\sigma$-finite measure on $(E,\mathcal E)$, 
recall that a process $\mathcal N$ is a Poisson point process on $E$ of intensity $\eta$ if
\begin{enumerate}
\item $\mathcal N$ is a point process (i.e. $\mathcal N=\sum_i\delta_{x_i}$ with $x_i$ $E$-valued random variables),
\item For every pairwise disjoint Borel sets $B_1,...,B_n\subset E$, the random variables
$\mathcal N(B_1),...,\mathcal N(B_n)$ are independent Poisson
random variables with respective parameters $\eta(B_1),...,\eta(B_n)$. 
\end{enumerate}
Let $M_p(E)$ be the space of all point measures defined on $E$, endowed with the topology of vague convergence; it is metrizable as a complete separable metric space.
A family of point processes $(\mathcal N_\eps)_\eps$ converges in distribution to $\mathcal N$ if for any bounded continuous function $f\colon  M_p(E)\to \RR$ the following convergence holds true 
\begin{equation}\label{cvdf}
\EE(f(\mathcal N_\eps))\to \EE( f(\mathcal N)),\quad \mbox{as }\varepsilon\rightarrow 0.
\end{equation}

For a collection $\mathcal A$ of measurable subsets of $\Omega$,
we define the following quantity:
\begin{equation}\label{defDelta}
\Delta(\A)
 := \sup_{A\in \A,B\in\sigma(\cup_{n=1}^{\infty}T^{-n}\A)} \left|\mu(A\cap B)-\mu(A)\mu(B)\right|.
\end{equation}

We set $\lambda$ for the Lebesgue measure on $[0,\infty)$.
\begin{theorem}(Convergence result for transformations \cite[Theorem 2.1]{ps20})\label{THM}
We assume that 
\begin{enumerate}
\item for any finite subset $\mathcal W_0$ of $\mathcal W$
we have $\Delta(H_\eps^{-1}\mathcal W_0)=o(\mu(A_\varepsilon))$,
\item there exists a measure $m$ on $(V,\mathcal V)$
such that for every $F\in\mathcal W$, $m(\partial F)=0$ and 
   $\lim_{\varepsilon\to 0}\mu(H_\eps^{-1}(F)|A_\eps)$ converges
 to $m(F)$.
\end{enumerate}
Then the family of
point processes $(\mathcal N_\varepsilon)_\varepsilon$
converges strongly\footnote{i.e. with respect to any probability measure absolutely continuous w.r.t. $\mu$} in distribution, as $\eps\rightarrow 0$,
 to a Poisson point process $\mathcal P$ of intensity $\lambda\times m$. 

In particular, for every
relatively compact open $B\subset E$ such that $(\lambda\times m)(\partial B)=0$,
$(\mathcal N_\varepsilon(B))_\varepsilon$ converges in distribution, as $\eps\rightarrow 0$,
to a Poisson random variable with parameter $(\lambda\times m)(B)$. 
\end{theorem}
\begin{theorem}(Convergence result for special flows  \cite[Theorem 2.3]{ps20})\label{THMflow}
Assume $(\Omega,\mu,Y=(Y_t)_t)$ can be represented as a special flow over
a probability preserving dynamical system $(M,\nu,F)$ with roof function $\tau:M\rightarrow(0,+\infty)$ with $M\subset \Omega$ and set $\Pi:\Omega\rightarrow M$ for the projection such that  $\Pi(Y_s(x))=x$ for all $x\in M$ and all $s\in[0,\tau(x))$.\\
Assume moreover that $Y$ enters at most once in $A_\varepsilon$ between two consecutive visits to $M$ and that there exists a family of measurable normalization functions $G_\varepsilon:M\rightarrow V$ such that the family of point processes
$(\mathcal N(F,\Pi(A_\varepsilon),h_\varepsilon,G_\varepsilon))_\varepsilon$
converges in distribution,  as $\eps\rightarrow 0$ and with respect to some probability measure $\tilde\nu\ll\nu$, to a Poisson point process of intensity $\lambda\times m$, where $m$ is some measure on $(V,\mathcal V)$, then the family of point processes $(\mathcal N(Y,A_\varepsilon,h_\varepsilon/\mathbb E_{\nu}[\tau],G_\varepsilon\circ\Pi))_\varepsilon$
converges in distribution, as $\eps\rightarrow 0$ (with respect to any probability measure absolutely continuous with respect to $\mu$), to a Poisson process $\mathcal P$ of intensity $\lambda\times m$.
\end{theorem}

\section{Number of visits to a small set before the first visit to a second small set}

Suppose $B_\eps^0$ and $B_\eps^1$ are two disjoint sets. We define the spatio-temporal process $\N_\eps$ with $A_\eps=B_\eps^0\cup B_\eps^1$, $H_\eps(x)=\ell$ if $x\in B_\eps^\ell$, $\ell=0,1$, that is on $[0,+\infty)\times \{0,1\}$
\begin{equation}\label{SPforhazard}
\N_\eps(x) = \sum_{n=1}^\infty \sum_{\ell=0}^1 \delta_{(n\mu(A_\eps),\ell)}1_{B_\eps^\ell}(T^n x)
\end{equation}
in the case of a transformation $T$ or 
\begin{equation}\label{SPforhazardflot}
\N_\eps(x) = \sum_{t>0} \sum_{\ell=0}^1 \delta_{(t h_\eps,\ell)}1_{Y_t\ enters\ B_\eps^\ell}
\end{equation}
in the case of a flow $Y$.
In \cite{kr} Kifer and Rapaport studied the distribution of a (multiple) event $T^nx\in B_\eps^1$ until a (multiple) hazard $T^n(x)\in B_\eps^0$.  We stick here to single event and hazard and define,
in the case of a transformation $T$,
\begin{equation}\label{hithazard}
\M_\eps(x) := \sum_{n=1}^{\tau_{B_\eps^0}(x)}1_{B_\eps^1}(T^nx)\, ,
\end{equation}
where we set $\tau_B(x):=\inf\{n\ge 1\, :\, T^n(x)\in B\}$ or, in the case of a flow $Y$:
\begin{equation}\label{hithazardflot}
\M_\eps(x) := \sum_{t_\in(0,\tau_{B_\eps^0}(x))}1_{Y_t\mbox{ enters }B_\eps^1}\, ,
\end{equation}
where we set $\tau_B(x):=\inf\{t>0\, :\, Y_t(x)\in B\}$.
The process $\mathcal M_\eps$ counts the number of entrances of the flow in the 1-set before 
its first visit to the 0-set.\\
In the case of a flow, it is also natural to consider the following process $\mathcal M'_\eps$
measuring the time spent by the flow in the 1-set before its first visit to the 0-set:
\begin{equation}\label{hithazardflot2}
\M'_\eps(x) :=\int_0^{\tau_{B_\eps^0}(x)}1_{B_\eps^1}\circ Y_s(x)\, ds\, .
\end{equation}
In view of the study of this last process, we will consider the following process measuring the time spent by the flow in each set:
\[
\left(\mathcal L_\eps:=\sum_{j=0}^1\sum_{t\, :\, Y_t\mbox{ enters }B_\eps^j}\delta_{th_\eps,j,a_\eps D_{B_\eps^j}\circ Y_t}\right)_{\eps>0}
\] 
with $D_A:=\tau_{\Omega\setminus A}$. 

\begin{theorem}\label{Thm2ensembles}
Let $p\in(0,1)$ and $\mathbb P$ be a probability measure on $\Omega$. 
Assume,
in the case of a flow, that $\lim_{\eps\rightarrow 0}\mathbb P(
B_\eps^0\cup B_\eps^1)$=0.\\

If the spatio-temporal process $\N_\eps$ defined as in \eqref{SPforhazard} or \eqref{SPforhazardflot} converges, with respect to $\mathbb P$, to a PPP of intensity $\lambda\times\B(p)$ where $\B(p)$ denotes the Bernoulli measure with parameter $p$ (for a transformation we expect $p=\lim_{\eps\to0} \mu(B_\eps^1)/\mu(A_\eps)$), then the process $(\M_\eps)_\eps$ has asymptotically geometric distribution, more precisely it converges in distribution to $\M$ 
with $\mathbb P(\M=k)=p^{k}(1-p)$ for any $k\ge 0$; in particular the asymptotic value for the \emph{commitor function} is
	\[
	\lim_{\eps\to0} \mathbb P(\tau_{B_\eps^0}<\tau_{B_\eps^1}) = \lim_{\eps\to0} \mathbb P(\M_\eps=0)=1-p.
	\]
In the case of a flow, if $(a_\eps\tau_{\Omega\setminus B^{1}_\eps})_\eps$
converges in probability $\mathbb P$ to 0 and if $(\mathcal L_\eps)_{\eps>0}$ supported on $[0,+\infty)\times\{0,1\}\times\bar\RR_+$ converges in distribution with respect to $\mathbb P$ to
a PPP $\mathcal L_0$ with intensity $\lambda\times \sum_{j=0}^1p_j(\delta_j\times m'_j)$
where the $m'_j$ are probability measures,
then $(a_\eps\M'_\eps)_\eps$
converges to $\sum_{i=1}^{\M}X_i$ where $(X_i)_i$ is a sequence of i.i.d. random variables with distribution $m'_1$ and independent of $\M$ where $\M$ is as above.
\end{theorem}

\begin{proof}
We first observe that the mapping 
\[
J:\xi\in M_p([0,+\infty)\times\{0,1\})\mapsto \xi([0,\tau^{0}]\times\{1\})
\]
is continuous
 where $\tau^0=\sup\{t\ge0\colon \xi([0,t]\times\{0\})=0\}$ is continuous at a.e. realization $\xi$ of $\chi:=PPP(\lambda\times\B(p))$. Indeed, $\xi(\cdot\times\{0\})$ and $\xi(\cdot\times\{1\})$ are the realization of two homogeneous independent Poisson process hence $\tau^0$ is a.s. not an atom of $\xi(\cdot\times\{1\})$.
Observe that, in the case of a transformation, $\M_\eps=J(\N_\eps)$ and in the case
of a flow $\mathbb P(\M_\eps\ne J(\N_\eps))=\mathbb P(Y_0\in B_\eps^0\cup B_\eps^1)\rightarrow 0$.
Therefore, by the continuous mapping theorem, $\M_\eps$ converges in distribution to $G:=J(\chi)$.

We now compute the law of $G$. The first hazard $\tau^0$ has an exponential distribution with parameter $1-p$, while $\chi^1(\cdot):=\chi(\cdot\times\{1\})$ is a Poisson point process with intensity $p\lambda$, and the two are independent. Therefore, for any $k\in\NN$
\[
\begin{split}
\PP(G=k) &= 
\PP(\chi^1([0,\tau^0])=k
)\\
& = \int_0^\infty e^{-pt} \frac{(pt)^k}{k!} (1-p)e^{-(1-p)t}\, dt = (1-p)p^k.
\end{split}
\]
This ends the proof of the first points of the Theorem. Let us now prove the last one.
We use the fact that the mapping $J:\xi\in M_p([0,+\infty)\times\{0,1\}\times\bar\RR_+)\mapsto  \int_{[0,\tau^0]\times\{1\}\times[0,K_0]} z \, d\xi(t,j,z)$
is continuous at a.e. realization $\xi$ of $\chi$ and conclude as above by the continuous mapping theorem and the Slutzky lemma since 
$a_\eps\M'_\eps=\mathbf 1_{\{Y_0\not\in B^0_\eps\}}\left(J(\L_\eps)+ a_\eps\tau_{\Omega\setminus B_\eps^1}\right)$.
\end{proof}
\begin{example}\label{exaSinai2boules}
Consider the billiard flow $(Y_t)_t$  associated to a Sinai billiard with finite horizon in a domain $Q\subset \mathbb T^2$ (see Appendix for details).
Let $\mathbb P$ be any probability measure on $\Omega:=Q\times S^1$ absolutely continuous with respect to Lebesgue.
We fix two distinct point positions $q_0,q_1\in Q$ and
two positive real numbers $r_0,r_1>0$. Set
$B_\eps^{i}:=B(q_i,r_i\eps)\times S^1$ and $d_i=2-\mathbf 1_{q_i\in\partial Q}$.\\ Then $(\M_\eps)_\eps$ converges in distribution with respect to $\mathbb P$ to  $\M$ 
with $\mathbb P(\M=k)=p^{k}(1-p)$ for any $k\ge 0$ and with $p=\frac{d_1r_1}{d_0r_0+d_1r_1}$.\\
Moreover $(\eps^{-1}\M'_\eps)_\eps$
converges in distribution with respect to $\mathbb P$ to $r_1\sum_{i=1}^{\M}Y_i$ where $(Y_i)_i$ is a sequence of i.i.d. random variables with density $y\mapsto \frac {y}{\sqrt{1-y^2}}\mathbf 1_{[0,1]}(y)$ independent of $\M$, with $\M$ as above.

\end{example}
\begin{proof}
Recall that the billiard flow $Y$ preserves the normalized Lebesgue measure $\mu$ on $Q\times S^1$. 
In view of applying Theorem~\ref{Thm2ensembles}, observe first that
$\lim_{\eps\rightarrow 0}\mathbb P(B_\eps^0\cup B_\eps^1)=0$ and $\mathbb E[\eps^{-1}\tau_{\Omega\setminus B_\eps^1}]\le 2r_1\mathbb P(B_\eps^1)$, thus $(\eps\tau_{\Omega\setminus B_\eps^1})_\eps$ converges in   probability $\mathbb P$ to $0$.\\

As a direct consequence of Theorem~\ref{thm:billplusieursboules}, 
the family of spatio-temporal processes $(\mathcal N_\eps)_{\eps>0}$ given by~\eqref{SPforhazardflot}, with $h_\eps=\frac{(d_0r_0+d_1r_1)\pi\eps}{Area(Q)}$, converges in distribution
to a PPP of intensity $\lambda\times\B(\frac{d_1r_1}{d_0r_0+d_1r_1})$ and so the
first conclusions of Theorem~\ref{Thm2ensembles} holds true with $p=\frac{d_1r_1}{d_0r_0+d_1r_1}$. This ends the proof of the convergence $(\M_\eps)_\eps$.\\
Due to Theorem~\ref{thmtempspassebillard}, $(\L_\eps)_\eps$ with $a_\eps=\eps$ and $h_\eps$ as previously converges in distribution to a PPP with intensity $\lambda\times \sum_{j=0}^1p_j(\delta_j\times m'_j)$ where $p_j:=\frac{d_jr_j}{d_0r_0+d_1r_1}$
and where $m'_j$ has density $y\mapsto \frac {y}{2r_j\sqrt{4r_j^2-y^2}}\mathbf 1_{[0,2r_j]}(y)$. Thus the last conclusion of Theorem~\ref{Thm2ensembles} holds also true with these notations. We conclude by taking $Y_i=X_i/(2r_1)$.
\end{proof}

\section{Number of high records}
We define the high records point process by 
\[
 \R_f(u,\ell)=\sum_{k=1}^\infty \delta_{k u} 1_{\{f\circ T^k>\max(\ell, f,...,f\circ T^{k-1})\}}\, .
\]
The successive times of records of an observable along an orbit are obviously tractable from the time and values of the observations along this orbit. The following proposition states that this is still the case for the corresponding asymptotic distributions. This has already been noticed in \cite{ht}, in particular in the context of Extremal events. Our result is similar to the proof of \cite[Theorem 3.1]{ht} from \cite[Theorem 5.1]{ht}.
\begin{proposition}\label{propRecords}
Let $(\Omega,\mathcal F,\mu,T)$ be a probability preserving dynamical system and $f:\Omega\rightarrow [0,+\infty)$ be a measurable function.
Assume the family 
 $\left(\mathcal N_{\varepsilon}=\mathcal N(T,\{f>\varepsilon^{-1}\},h_\varepsilon,
1/(\eps f)
)\right)_{\varepsilon>0}$
of point processes
on $[0,+\infty)\times[0,1]$
 converges in distribution with respect to $\mathcal P$ to a Poisson point process of  intensity $\lambda\times m$ with $m$ a probability measure on $[0,1]$ without any atom.
Then $\left(\R_f( h_\eps,\varepsilon^{-1})\right)_{\varepsilon>0}$
converges in distribution, as $\varepsilon\rightarrow 0$ to a Point process $\R=\sum_{\ell=1}^\infty Z_\ell \delta_{T_\ell}$ where $T_\ell=\sum_{i=1}^\ell X_i$, the $X_i$ are independent standard exponential random variable and the $Z_\ell$ are independent random variable of Bernoulli distribution with respective parameters $\ell^{-1}$, and the two sequences are independent.
\end{proposition}
\begin{proof}
Define the mapping 
\[
F \colon \xi=\sum_i \delta_{(t_i,v_i)}\in M_p([0,\infty)\times[0,1])\mapsto \sum_{i\in I(\xi)}\delta_{t_i},
\]
where $I(\xi)$ are the records of $\xi$, defined by those $i$ such that for any $j$ one has $t_j<t_i \implies v_j>v_i$. 
The map $F$ is continuous at each $\xi$ such that the $t_i$'s, and the $v_i$'s, are distincts.
This is the case for a.e. realization $\xi$ of a Poisson process of intensity $\lambda\times m$.
Therefore by the continuous mapping theorem $\R_f(h_\eps,\eps^{-1})=F(\N_\eps)$ converges to $\chi=F(PPP(\lambda\times m))$. 

We are left to compute the distribution. Observe that $PPP(\lambda\times m)$ is distributed as $\sum_{\ell=1}^\infty \delta_{(T_\ell,W_\ell)}$ with $(T_\ell)$ as in the statement and the $W_\ell$ are i.i.d. with distribution $m$, the two sequences being independent. Let $Z_\ell=1_{\{W_\ell \text{ is a record}\}}$. By \cite[Proposition 4.3]{resnick} the $Z_\ell$ are independent, have probability $1/\ell$, and when $Z_\ell=1$ we keep the point $T_\ell$.
\end{proof}
In particular, for every $t>0$ the number of records exceeding the value $\eps^{-1}$ before the time $th_\eps^{-1}$ corresponds to $\R_f(h_\eps,\eps^{-1})([0,t])$ and the conclusion of Proposition~\ref{propRecords} implies that it converges to $\sum_{\ell=1}^{N_t}Z_\ell$
where $Z_\ell$ are as in Proposition~\ref{propRecords} and where $(N_s)_s$ is a standard Poisson Process independent of $(Z_\ell)_\ell$.
\begin{example}
Consider a dispersive billiard with corner and cusps of maximal order $\beta_*>2$ as in \cite{jpz}. Consider the induced system $(\Omega,\mu,T)$ corresponding to the successive reflection times outside a neighbourhood $\mathcal U$ of cusps and write $R(x)$ for the number of reflections in $\mathcal U$ starting from $x$. Set $\alpha=\frac{\beta_*}{\beta_*-1}\in(1,2)$.\\
Setting $A_\eps:=\{R\circ T^{-1}>\eps^{-1}\}$, it has been proved in~\cite[Lemma 4.5]{jpz} that there exists an explicit $c_0>0$ such that $\mu(A_\eps)\sim c_0 
\eps^\alpha$ as
$\eps\rightarrow 0$.\\ 
The  assumptions of Proposition~\ref{propRecords} hold true with
$f=R\circ T^{-1}$ and $h_\eps=\mu(A_\eps)\sim c_0
\eps^\alpha$.
So the same assumptions hold true with $h_\eps= c_0 
\eps^\alpha$.\\

Furthermore the number $R_n$ of records of $R$ higher than $n^{1/\alpha}$ before the $n$-th reflection outside cusps converges to  
$\sum_{\ell=1}^{N}Z_\ell$
where $Z_\ell$ are as in Proposition~\ref{propRecords} and where $N$ is a Poisson random variable of parameter $c_0$ and independent of $(Z_\ell)_\ell$.
\end{example}
\begin{proof}
It follows from the proof of~\cite[Lemma 4.8]{jpz} that\footnote{\cite[Lemma 4.8]{jpz} states that this convergence is true in the set of point processes on $[0,+\infty)\times[1,+\infty)$, but its proof can be adapted in a straighforward way to obtain our purpose by considering not only intervals of the form $(c,c')$
but also intervals of the form $(c,+\infty]$.}
the family of point processes
$(\mathcal N(T,A_\eps,\mu(A_\eps),\eps R\circ T^{-1}))_\eps$ on $[0,+\infty)\times[1,+\infty]$ converges in distribution to a PPP  with intensity 
of density $(t,y)\mapsto \alpha
y^{-\alpha-1}\mathbf 1_{y>0}$ with respect to the Lebesgue measure.\\
Therefore the  assumptions of Proposition~\ref{propRecords} hold true with
$f=R\circ T^{-1}$ and $h_\eps=\mu(A_\eps)\sim c_0
\eps^\alpha$.
So the same assumptions hold true with $h_\eps= c_0 
\eps^\alpha$. This ends the proof of the first part.

For the second we apply Proposition~\ref{propRecords} with $\eps=n^{-\frac 1\alpha}$.
\end{proof}

\section{Line process of random geodesics}
We study the line process generated by a geodesic as in \cite{line} and recover their main result.
Let $N$ be a compact Riemannian surface of negative curvature. The geodesic flow $(Y_t)_t$ on the unit tangent bundle $\Omega=T^1N $ preserves the Liouville measure $\mu$.
Let $\pi_N\colon T^1N\to N$ be the canonical projection $(q,v)\mapsto q$.
We denote by $D(q,\eps)$ the ball in $N$ of radius $\eps$.
We now state the main theorem, postponing the details and precise definitions thereafter.
\begin{theorem}\label{thm:line}
	Fix $q_0\in N$. For any $a>0$, the intersection of the neighborhood $D(q_0,\eps)$ with the geodesic segment $\pi_N(\{ Y_t(x),0\le t\le a\eps^{-1} \})$, where $x$ is taken at random on $(\Omega,\mu)$, converges in distribution, after normalization,
	as $\eps\to0$, 
	to a Homogeneous Poisson line process in the unit disk of intensity $a/Area(N)$.
\end{theorem}

A Poisson line process in the unit disk $D$ of the plane, of intensity $\kappa\in(0,\infty)$, is a probabilistic process which draw lines in the disk. Each line $L$ is parametrized by $(r,\theta)\in[-1,1]\times[0,\pi]$ where
\[
L=\{(x,y)\in D\colon r=x\cos\theta+y\sin\theta\},
\]
and the parameters $(r,\theta)$ are produced by a Poisson point process of intensity $\frac\kappa\pi dr d\theta$ on $[-1,1]\times [0,\pi]$.
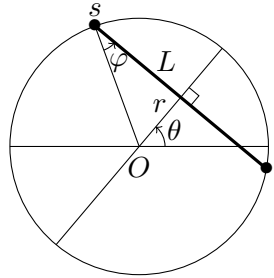
\begin{figure}[h!]
	\begin{tikzpicture}[scale=1.7]
	\coordinate (O) at (0,0);
	\coordinate (A) at (110:1);
	\coordinate (B) at (350:1);
	\coordinate (C) at (50:1);
	\coordinate (D) at (230:1);
	\draw[below] (O) node {$O$};
	\draw (A) node {$\bullet$};
	\draw (B) node {$\bullet$};
	\draw (O) circle (1);
	\draw (O) -- (A);
	\draw (C) -- (D);
	\draw[very thick] (A) -- (B);
	\draw (180:1) -- (0:1);
	\draw[->] (0.2,0) arc (0:50:0.2);
	\draw (25:0.3) node {$\theta$};
	\draw[above] (50:0.25) node {$r$};
	\draw[above] (A) node {$s$};
	\draw[->] (110:0.8) arc (290:320:0.2);
	\draw (110:0.7)+(0.08,0) node {$\varphi$};
	\draw (50:0.6) -- ++(320:0.1) -- ++(230:0.1);
	\draw[above] (A)+(322:0.7) node {$L$};
	\end{tikzpicture}
	\caption{Parametrization of the line $L$ by $(r,\theta)$ or $(s,\varphi)$.}\label{pardisk}
\end{figure}
Equivalently, changing the parametrization to $(s,\varphi)$ where $s\in\partial D=:S$ is one point of intersection of the line with the unit circle and $\varphi$ is the angle between the line $L$ (directed into the disk) and the normal at $s$ pointing inside the disk (see Figure~\ref{pardisk}), gives a Poisson point process of intensity $\frac{\kappa\cos\varphi}{2\pi}ds d\varphi$ (the jacobian is $\cos\varphi$ and each line has two representations in this parametrization). The intensity $\kappa$ in the theorem is equal to $a / Area(N)$, therefore the intensity in this parametrization will be $\frac{a}{2\pi Area(N)}\cos\varphi dsd\varphi = \frac{a}{Vol(T^1N)}\cos\varphi ds d\varphi$. The convergence of a point process in this parametrization implies it in the original one (by continuity of the change of parameter; see \cite[Proposition 3.18]{resnick}).

The exponential map $\exp_{q_0}$ is a local diffeomorphism on a neighborhood $U\subset T_{q_0}N$ of $0$. Thus its inverse is well defined on $D(q_0,\eps)$ for $\eps$ small enough so that $B(0,\eps)\subset U$. We identify $T_{q_0}N$ with $\RR^2$. Set $V=S\times [-\frac\pi2,\frac\pi2]$. 
For $q\in D(q_0,\eps)$ we let $s_\eps(q) = \eps^{-1}\exp_{q_0}^{-1}(q)$ and for $q\in \partial D(q_0,\eps)$ and $v\in T_qN$ we denote by $\phi_q(v)$ the angle between the normal at $q$ pointing inside the disk and $v$ (see Figure~\ref{intball}).
\begin{figure}[h!]
	\begin{tikzpicture}[scale=1.7]
	\coordinate (Oa) at (0,0);
	\coordinate (Aa) at (110:1);
	\coordinate (Ba) at (330:1);
	\draw[below] (Oa) node {$q_0$};
	\draw (Aa) node {$\bullet$};
	\draw (Ba) node {$\bullet$};
	\draw (Oa) node {$\bullet$};
	\draw (Oa) circle (1);
	\draw (Oa) -- (Aa);
	\draw[very thick] (Aa) to [bend left=10] (Ba);
	\draw[->,dashed,very thick] (Aa) to (50:0.5);
	\draw[above] (Aa) node {$q$};
	\draw[->] (110:0.8) arc (290:320:0.2);
	\draw[left] (110:0.75) node {$\phi_q( v)$};
	\draw[above] (Aa)+(322:0.5) node {$ v$};
	\draw[right] (0.5,0.1) node {$\gamma$};
	\end{tikzpicture}
	\caption{A geodesic arc $\gamma$ entering the ball $D(q_0,\eps)$.}\label{intball}
\end{figure}
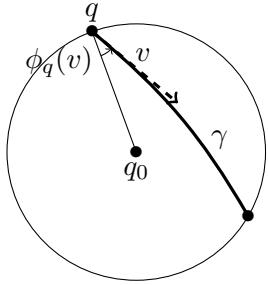

The intersection $\I_\eps^a(x):=\pi_N(Y_{[0,a\eps^{-1}]}(x))\cap D(q_0,\eps)$ consists of finitely many geodesic arcs $\gamma_i:=\pi_N(Y_{[t_i,t_i+\ell_i]}(x))$, where $\ell_i$ is the length of the arc; we drop the dependence on $x$ and $\eps$ for simplicity. 
 The arcs $\gamma_i$ are fully crossing the ball, except possibly for the two extremities (at $t=0$ or $t=a\eps^{-1}$) which could give an incomplete arc. The later happens with a vanishing probability as $\eps\to0$, therefore we will ignore this eventuality. The arc $\gamma_i$ enters the ball at the position $q_i$ with direction $v_i$ where $(q_i,v_i):=Y_{t_i}(x)$. 

When $\eps\to 0$, the geodesic arcs $\gamma_i$ which compose the intersection $\I_\eps^a$ become more and more straight. This justifies the definition of the convergence in distribution of $\I_\eps^a$ as the convergence in distribution of the point process 
\begin{equation}\label{ppsphi}
\sum_i \delta_{(s_\eps(q_i),\phi_q(v_i))}.
\end{equation}
Loosely speaking, we identify the images $s_\eps(\gamma_i)$ with the chord of the unit disk $D$ originated in $s_\eps(q_i)$ and direction $v_i$.

We now proceed with the proof of the theorem.
Let $\A_\eps \subset T^1N$ be the set of points $(q, v)$ such that $q\in\partial D(q_0,\eps)$ and $ v$ is pointing inside the ball.
We define on $\A_\eps$
\begin{equation}
\H_\eps(q,v)=(s_\eps(q),\phi_q(v)) \in V.
\end{equation}

The theorem is a byproduct of the following result for the geodesic flow.
\begin{proposition}\label{pro:boulepositionflot}
The process of entrances in the ball for the position for the geodesic flow  $\N(Y,\A_\eps,2\eps/Area(N),\H_\eps)$ on $[0,+\infty[\times V$
converges to a Poisson point process with intensity $\frac1{4\pi}\cos\varphi dt ds d\varphi$.
\end{proposition}

\begin{proof}[Proof of Theorem~\ref{thm:line}]
The counting process
\begin{equation}\label{box}
\L_\eps^a(\cdot):=\N(Y,\A_\eps,2\eps/Area(N),\H_\eps)([0,2a/Area(N)]\times \cdot)
\end{equation}
produces a point $(s,\varphi)$ each time that the geodesic flow $Y_t$ enters in $D(q_0,\eps)$ for some $t$ such that $2\eps t /Area(N) \le 2a/Area(N)$, that is $t\le a\eps^{-1}$. 
By Proposition~\ref{pro:boulepositionflot} and the continuous mapping theorem the point process $\L_\eps^a$ converges to a Poisson point process of intensity 
$\frac{2a}{Area(N)}\frac1{4\pi}\cos\varphi ds d\varphi$. 
By the above discussion, in particular~\eqref{ppsphi}, this completes the proof of the theorem.
\end{proof}

We emphasize that this proof only uses the convergence stated in Proposition~\ref{pro:boulepositionflot}, therefore it applies for more general 'geodesic-like flows', for instance the argument applies immediately to billiards systems, using Theorem~\ref{thm:billplusieursboules} in place of Proposition~\ref{pro:boulepositionflot}.

\begin{proof}[Proof of Proposition~\ref{pro:boulepositionflot}]
The first step is to construct a Markov section for the geodesic flow, subordinated to a finite family of disks $D_i\subset T^1N$. Fix some $\delta>0$ sufficiently small. 
By Bowen \cite{bowen} there exists a Markov section $(X_i)_i$ of size $\delta$, in particular $\diam X_i<\delta$ and $T^1N=\cup_i Y_{[-\delta,0]}(X_i)$. One can choose the disks $D_i\supset X_i$ in such a way that 
$$D_i\subset \{(q, v)\colon q\in Q_i, |\angle( n_q, v)|>\frac{\pi}{2}-\delta\}$$
where $Q_i$ are $C^2$ curve in $N$ and $n_q$ is the normal vector to $Q_i$ at $q$ (with $q\mapsto  n_q$ continuous). Without loss of generality we assume that $q_0\not\in \cup_i Q_i$.

The flow $(Y_t)$ is represented by a special flow over the Poincar\'e section $M:=\cup_i X_i$, with a $C^2$ roof function $\tau$. Let $\Pi$ be the projection onto $M$ along the flow in backward time. The flow $(T^1N,(Y_t),\mu)$ projects down to a system $(M,F,\nu)$, conjugated to a subshift of finite type with a Gibbs measure of a H\"older potential. In order to apply Theorem~\ref{THMflow} we need to check that the set $A_\eps:=\Pi\A_\eps$ and $H_\eps(x):=\H_\eps(Y_s(x))$ where $s>0$ is the minimal time such that $Y_s(x)\in \A_\eps$ fulfills the hypotheses of Theorem~\ref{THM}.
For that we will apply \cite[Proposition 3.2]{ps20}. The Poincar\'e map $F$ has a hyperbolic structure with an exponential rate, thus it satisfies the setting of \cite[Proposition 3.2]{ps20} with any polynomial rate $\alpha$, in particular $\alpha=4$ works. Here the boundary is meant in the induced topology on $M$. It suffices to prove that for some $p_\eps=o(\nu(A_\eps))$ one has 
(i) $ \nu(\tau_{A_\eps}\le p_\eps) = o(1)$ and (ii) $\nu((\partial A_\eps)^{[p_\eps^{-\alpha}]})=o(\nu(A_\eps))$, the two other assumptions being trivially satisfied in our situation. 

Measure of $A_\eps$:
The Liouville measure $\mu$ is the product of the normalized surface on $N$ times the Haar measure on $T^1N$. Its projection $\nu$ to the Poincar\'e section satisfies $d\nu = c_\nu \cos\varphi dr d\varphi$ for some normalizing constant $c_\nu=\left(\sum_i \int_{X_i}\cos\varphi drd\varphi\right)^{-1}$, where $r$ is the curvilinear abscissa on $Q_i$ and $\varphi$ the angle between the velocity and the normal to $Q_i$.
Moreover we have $d\mu = (\int_M\tau d\nu)^{-1}d\nu\times dt|_{M_\tau}$ where $M_\tau=\{(x,t)\colon x\in M,0\le t<\tau(x)\}$.

The geodesic flow preserves the measure $\cos\varphi drd\varphi$ from $A_\eps\subset M$ to $\A_\eps$, therefore 
\[
\begin{aligned}
\nu(A_\eps)&=c_\nu \int_{A_\eps}\cos\varphi drd\varphi = c_\nu \int_{\A_\eps} \cos\varphi dr d\varphi = c_\nu\int_{\partial D(q_0,\eps)} dr \int_{-\pi/2}^{\pi/2}\cos\varphi d\varphi \\
&\sim c_\nu 4\pi\eps.
\end{aligned}
\]

Short returns:
For any $q\in D(q_0,\eps)$,
let $R_\eps(q)$ be the set of $ v\in T^1_qN$ such that 
the geodesic segment $\gamma_{[0,\eps^{-1/2}]}(q, v)$ enters again $D(q_0,\eps)$ after leaving $D(q_0,2\eps)$. The result of \cite[Lemma 5.3]{line} ensures the existence of $K>0$ such that
for any $q\in D(q_0,\eps)$
\[
Leb (R_\eps(q))   \le K \eps^{1-1/2} = K\sqrt{\eps}.
\]
Therefore, setting $\hat\A_\eps = \{(q,v)\in\A_\eps\colon v\in R_\eps(q)\}$ we get that the bidimensional Lebesgue measure of $\hat\A_\eps$ is $O(\eps^{3/2})$. A fortiori since the projection $\Pi$ preserves the measure $\cos\varphi dr d\varphi$ we get 
\[
\nu(\Pi \hat \A_\eps)= c_\nu \int_{\Pi\hat\A_\eps} \cos\varphi dr d\varphi =
c_\nu \int_{\hat\A_\eps} \cos\varphi dr d\varphi = O(\eps^{3/2}).
\] 
Let $p_\eps=\lfloor(\max \tau)^{-1} \eps^{-1/2}\rfloor$ and notice that
 $A_\eps\cap \{\tau_{A_\eps}\le p_\eps\} \subset \Pi\hat \A_\eps$.
By the previous estimates we get
\[
\nu( A_\eps \cap \{\tau_{A_\eps}\le p_\eps\}) =O( \eps^{3/2}).
\]
Hence
\[
\nu(\tau_{A_\eps}\le p_\eps|A_\eps)=o(1).
\]
This is the assumption (i).

We now prove (ii).
The boundary of $A_\eps$ in the induced topology of $M$ is included in the set of $\Pi(q, v)$ where $ v$ is tangent to the boundary of $\partial D(q_0,\eps)$. This defines for each $i$ such that $X_i\cap A_\eps$ is nonempty at most two $C^2$ curves in $D_i$ of finite length (by transversality), therefore its $\eps^2$-neighborhood has a measure $O(\eps^2)$.

Finally, the measure $dm_\eps=(H_\eps)_*\nu(\cdot|A_\eps)$ is equal to the measure $dm:=\frac{1}{4\pi}\cos\varphi ds d\varphi$, since the measure $\cos\varphi dr d\varphi$ is preserved by the inverse of the projection $\Pi$ from $A_\eps$ to $\A_\eps$ and $\H_\eps$ has constant jacobian $\eps$ in these coordinates. By Theorem~\ref{THM} the point process $\N(F,A_\eps,\nu(A_\eps),H_\eps)$ converges to a Poisson point process of intensity $\lambda\times m$.
Applying Theorem~\ref{THMflow} with $h_\eps=c_\nu4\pi\eps$ and $h_\eps'=h_\eps/E_\nu(\tau)$ we get that $\N(Y,\A_\eps,h_\eps',\H_\eps)$ converges to a Poisson point process of intensity $\lambda\times m$.
In addition, 
\[
\int_M \tau d\nu = c_\nu \int_M \tau \cos\varphi dr d\varphi = c_\nu \int_{M_\tau} \cos\varphi dt dr d\varphi = c_\nu Vol(T^1N).
\]
Thus, since $Vol(T^1N) = 2\pi Area(N)$ we get that $h_\eps'=\frac{2\eps}{Area(N)}$, proving the proposition.
\end{proof}

\section{Time spent by a flow in a small set}
Given a flow $Y=(Y_t)_t$ defined on $\Omega$ and a set $A\subset\Omega$, a very natural question is to study the time spent by the flow in the set $A$, that is the local time $L_T(A)$ given by following quantity~:
\[
L_T(A):=\lambda\left(\left\{t\in[0,T]\, :\, Y_t\in A\right\}\right)\, .
\]
This quantity measures the time spent by the flow $Y$ in the set $A$ between time 0 and time $T$ (the symbol $L$ refers to the local time).
We also write $D_A:=\inf\{t>0\, :\, Y_t\not\in A\}$ for the duration of the present visit to the set $A$.

\begin{proposition}\label{tempspasseflot}
Let $J\ge 1$ and $Y=(Y_t)_{t\ge 0}$ be a 
flow defined on $(\Omega,\mathcal F,
\mathbb P
)$.
Assume that $\left( \mathcal N_\varepsilon=N(Y,A_\varepsilon,h_\varepsilon,H_\varepsilon)\right)_{\varepsilon>0}$
converges in distribution (with respect to $\mathbb P$) to a PPP $\mathcal N_0$ of intensity $\lambda\times m$ with $H_\eps(A_\eps)\subset V=\{1,...,J\}\times W$ where $m=\sum_{j=1}^J(p_j\delta_{j}\times m_j)$, with $\sum_{j=1}^Jp_j=1$ and where $m_j$ are probability measure on some separable metric space $W$. Suppose in addition that, for some $a_\eps$ and each $x$ entering in $A_\varepsilon$, $a_\varepsilon D_{A_\varepsilon}(x)= \mathfrak D_\varepsilon( H_\varepsilon(x))$ with $\lim_{\varepsilon\rightarrow 0} \mathfrak D_\varepsilon(j,w)=: \mathfrak D_j(w)$ uniformly in $w\in W$, where $\mathfrak D_j:W\rightarrow\bar\RR_+$ is continuous.

Then
\[
\left(\mathcal L_\eps=\sum_{t\, :\, Y_t(x)\mbox{ enters }A_\eps}\delta_{th_\eps,H_\eps^{(1)}(Y_t(x)),a_\eps D_{A_\eps}\circ Y_t(x)}\right)_{\eps>0}
\]
converges in distribution with respect to $\mathbb P$ to a PPP $\mathcal L_0$ on $[0,+\infty)\times\{1,...,J\}\times\bar\RR_+$ with intensity $\lambda\times\sum_{j=1}^J( 
p_j\delta_j\times (\mathfrak D_j)_*(m_j))$.

If moreover 
$a_\eps D_{A_\eps}\stackrel{\mathbb P}\rightarrow 0$,
 then, for every $T>0$,  $((a_\varepsilon L^{(j)}_{\lfloor t/h_\varepsilon \rfloor}(A_\eps))_{t\in[0,T],j=1,...,J})_{\varepsilon>0}$ converges in distribution to $\left(\sum_{k=1}^{N_t^{(j)}}X_k^{(j)}\right)_{t\in[0,T],j=1,...,J}$ as $\varepsilon\rightarrow 0$, where $(N_t^{(j)})_{t>0}$ are independent Poisson process with parameter $p_j$ and
where $(X_k^{(j)})_{k\ge 1}$ are independent sequences of independent identically distributed
random variables with distribution $\sum_{j=1}^Jp_j(\mathfrak D_j)_*(m_j)$ independent of $(N_t^{(j)})_{t>0}$
\end{proposition}
\begin{proof}
Observe that, for every $\epsilon\ge 0$, $\mathcal L_\eps=(\psi_\eps)_*(\mathcal N_\eps)$
with $\psi_\eps:(t,j,w)\mapsto (t,j,\mathfrak D_\eps(j,w))$ if $\eps>0$
and with $\psi_0:(t,j,w)\mapsto (t,j,\mathfrak D_j(w))$.
Using \cite[Proposition 3.13]{resnick} we prove the first statement.

Assume now that 
$a_\eps D_{A_\eps}\stackrel{\mathbb P}\rightarrow 0$.
Then 
\[
\left(a_\varepsilon L^{(j)}_{t/h_\varepsilon}=
a_\eps\min(t/h_\eps, D_{A_\eps})+
\int_{[0,t]\times\{j\}\times\bar\RR_+}z\, d\mathcal L_\eps(s,i,z)\right)_{t\in[0,T],j=1,...,J}
\]
which converges to $\left(a_\varepsilon L^{(j)}_{t/h_\varepsilon}=\int_{[0,t]\times\{j\}\times\bar\RR_+}z\, d\mathcal L_0(s,i,z)\right)_{t\in[0,T],j=1,...,J}$.
\end{proof}

We apply the previous result to the dispersive billiard flow in a Sinai billiard with finite horizon.\\

\begin{theorem}[Time spent by the billiard flow in a shrinking ball for the position]\label{thmtempspassebillard}
Consider the billiard flow associated to a Sinai billiard with finite horizon in a domain $Q\subset \mathbb T^2$ (see Appendix for details). Recall that this flow preserves the normalized Lebesgue measure on $Q\times S^1$. Let $J$ be a positive integer.
Let $q_1,...,q_J\in Q$ 
be a $J$ pairwise distinct fixed position in the billiard domain and $r_1,...,r_J$
be $J$ positive real numbers. 
We set $d_j=2$ if $q_j\not\in\partial Q$
and $d_j=1$ if $q_j\in\partial Q$ and $d:=\sum_{j=1}^Jd_jr_j$ and also
\[
\left(\mathcal L_\eps=\sum_{j=1}^J\sum_{t\, :\, Y_t(x)\mbox{ enters }B(q_j,r_j\eps)\times S^1}\delta_{\frac{d\pi\eps t}{Area(Q)},j,\eps^{-1} D_{B(q_j,r_j\eps)\times S^1}\circ Y_t(x)}\right)_{\eps>0}
\]
and
\[
L^{(j)}_{t/\eps}:=\int_0^{\frac t\eps} \mathbf 1_{\{Y_s(\cdot)\in  B(q_j,r_j\eps)\times S^1\}} \, ds\, .
\]
Then, $(\mathcal L_\eps)_{\eps>0}$ converges strongly in distribution to
a PPP $\mathcal L_0$ with intensity $\lambda\times\sum_{j=1}^J 
\frac {d_j r_j}{d}(\delta_j\times m'_j)$
where $m'_j$ is the distribution of $r_jX$ with $X$ a random variable of density $y\mapsto \frac{y}4.\arccos'(\frac y2)\mathbf 1_{[0,2]}(y)=\frac {y}{2\sqrt{4-y^2}}\mathbf 1_{[0,2]}(y)$.\\
Moreover, for every $T>0$, $(( \eps^{-1}L ^{(j)}_{t/\eps})_{t\in[0,T],j=1,\ldots,J})_{\varepsilon>0}$ converges strongly in distribution to $\left(r_j\sum_{k=1}^{N^{(j)}_{t}}X^{(j)}_k\right)_{t\in[0,T],j=1,\ldots,J}$ as $\varepsilon\rightarrow 0$, where $(N^{(j)}_t)_{t>0}$ are independent Poisson process with parameter $\frac{d_j\, Area(Q)}{d^2\pi}$,
where $(X_k)_{k\ge 1}$ is a sequence of independent identically distributed
random variables with density $x\mapsto \frac {x}{2\sqrt{4-x^2}}\mathbf 1_{[0,2]}(x)$ independent of $(N_t)_{t>0}$.
\end{theorem}
\begin{proof}[Proof of Theorem~\ref{thmtempspassebillard}]
Due to Theorem~\ref{thm:billplusieursboules}, we know that
the family of processes
$$\sum_{j=1}^J\sum_{t\ :\ (Y_s(y))_s\mbox{ enters } B(q_j,\varepsilon)\times S^1\mbox{ at time t}}\delta_{\left(\frac{d\pi\varepsilon t}{Area(Q)},\frac{\Pi_Q(Y_t(y))-q_j}\eps,\Pi_V(Y_t(y))\right)}$$
converges in distribution (when $y$ is distributed with respect to any probability measure absolutely
continuous with respect to the Lebesgue measure on $\mathcal M$) as $\eps\rightarrow 0$ to 
a Poisson Point Process with intensity $\lambda\times \tilde m_0$
where $\tilde m_0$ is the probability measure on $\{1,...,J\}\times S^1\times S^1$ with density $(j,p,\vec u)\mapsto\sum_{j=1}^J\frac {r_jd_j}{d}\frac{1}{2d_j\pi} \langle(-p),\vec u\rangle^+
\mathbf 1_{\{\langle p,\vec n_{q_j}\rangle\ge 0\}}$ with $d:=\sum_{j=1}^Jd_jr_j$.\\
We will apply  Proposition~\ref{tempspasseflot} with $A_\eps:=\bigcup_{j=1}^JB(q_j,\eps r_j)\times S^1$ and $H_\eps(q,\vec v)=\left(j,\frac{\overrightarrow{q_jq}}{r_j\eps},\vec v\right)$ if $q\in \partial B(q_j,r_j\eps)$.\\
Let $x=(q,\vec v)$ entering in $B(q_j,\eps)\times S^1$.
If the billiard flow crosses $B(q_J,\eps)\times S^1$ before any collision off $\partial Q$, then
\[
\eps^{-1}D_{B(q_j,r_j\eps)\times S^1}(q,\vec v)=2\eps^{-1}\widehat{(\overrightarrow{qq_j},\vec v)}=D_0(H_\varepsilon (x))\, ,
\]
with $D_0(j,p,\vec u)=2r_j\widehat{(-p,\vec u)}$. This is always the case if $q_j\not\in \partial Q$. But, if $q_j\in\partial Q$, it can also happen that the billiard flow collides $\partial Q$
at a point $q'\in B(q_j,\eps)$ before exiting $B(q_j,\eps)\times S^1$. Then
the point $q'$ is at distance in $\mathcal O(\eps^2)$ of the tangent line to $\partial Q$ at $q_j$, and the tangent line of $\partial Q$ at $q'$ makes an angle in $\mathcal O(\eps)$ with the tangent line of $\partial Q$ at $q_j$. In this case 
\[
\eps^{-1}D_{B(q_j,r_j\eps)\times S^1}(q,\vec v)=2\eps^{-1}\widehat{(\overrightarrow{qq_j},\vec v)}+\mathcal O\left(\eps\right)=D_0(H_\varepsilon (x))+\mathcal O\left(\eps\right)\, ,
\]
uniformly in $x=(q,\vec v)$ and $\eps$.
In any case, we set $a_\varepsilon=\eps^{-1}$ and  $\mathfrak D_\varepsilon=D_0+\mathcal O(\eps)$.\\
Applying now Proposition~\ref{tempspasseflot}, we infer that
$(\mathcal L_\eps)_{\eps>0}$ converges strongly in distribution to
a PPP $\mathcal L_0$ with intensity $\lambda\times\sum_{j=1}^J 
\frac {d_j r_j}{\sum_{j'=1}^Jd_{j'} r_{j'}}(\delta_j\times (D_j)_*(m_j))$, with
$D_j(p,\vec u)=2r_j\widehat{(-p,\vec u)}$ and
$m_j$ the probability measure on $S^1\times S^1$ with density
\[
(p,\vec u)\mapsto \frac{d_j}{2\pi} \langle(-p),\vec u\rangle^+
\mathbf 1_{\{\langle p,\vec n_{q_j}\rangle\ge 0\}}\, .
\]
It remains to identify the distribution $(D_j)_*(m_j)$.
By the transfer formula, we obtain
\begin{align*}
\int_0^\infty \!\!\!\!\! h(D_j(p,\vec u))\, dm_j(p,\vec u)&= \frac1{2d_j\pi}
\int_{S^1\times S^1}\!\!\!\!\! h(r_j\langle -2p,\vec u\rangle) \langle(-p),\vec u\rangle^+
\mathbf 1_{\{\langle p,\vec n_{q_j}\rangle\ge 0\}}\, dp\, d\vec u\\
&=\frac1{2}
\int_{-\frac\pi 2}^{\frac\pi 2} h(2r_j\cos\varphi) \cos\varphi\, d\varphi\\
&=\int_{0}^{\frac\pi 2} h(2r_j\cos\varphi) \cos\varphi\, d\varphi\\
&=\int_{0}^{2} h(r_jy) (\arccos(\cdot/2))'(y)\frac y2\, dy\, .
\end{align*}
Thus we have proved that the probability distribution $(D_j)_*\tilde m_j$ is the distribution of $r_jX$ with $X$ a random variable of density $y\mapsto \frac{y}4.\arccos'(\frac y2)\mathbf 1_{[0,2]}(y)=\frac {y}{2\sqrt{4-y^2}}\mathbf 1_{[0,2]}(y)$.\\
We can apply the last point of Proposition~\ref{tempspasseflot} 
since $\eps^{-1} D_{A_\eps}\le 2\max_j r_j \mathbf 1_{A_\eps}\stackrel{\mathbb P}\rightarrow 0$ for any probability measure $\mathbb P$ absolutely continuous with respect to the Lebesgue measure on $Q\times S^1$.

\end{proof}
\begin{appendix}
\section{Visits of the Sinai billiard flow to a finite union of balls for the position}\label{appendbill}
In this appendix we are interested in spatio temporal processes for the
Sinai billiard flow with finite horizon.\\
Let us start by recalling the model and introducing notations. 
We consider a finite family $\{O_i,\ i=1,...,I\}$ of convex open sets of the two-dimensional torus $\mathbb T^2=\mathbb R^2/\mathbb Z^2$. We consider the billiard domain $Q=\mathbb T^2\setminus\bigcup_{i=1}^IO_i$ and call the $O_i$ obstacles. We assume that
these obstacles have $C^3$-smooth boundary with non null curvature and that
their closures are pairwise disjoint. We consider a point particle moving in $Q$ in the following way: the point particle goes straight at unit speed in $Q$ and obeys to the classical Descartes reflexion law when it collides an obstacle. We then define the billiard flow $(Y_t)_{t\in\mathbb R}$ as follows. $Y_t(q,\vec v)=(q_t,v_t)$ is
the couple position-velocity of the point particle at time $t$ if the particle has position $q$ and velocity $\vec v$ at time 0.  To avoid any confusion, we consider that
the billiard flow is defined on the quotient $(Q\times S^1)/\mathcal R$, with $\mathcal R$ is the equivalence relation corresponding to the identification of pre-collisional and post-collisional vectors at a reflection time:
\[
(q,\vec v)\mathcal R(q',\vec v')\ \Leftrightarrow\ (q,\vec v)=(q',\vec v')\quad\mbox{or}\quad 
\vec v'=\vec v-2\langle \vec n_q,\vec v\rangle \vec n_q
\, ,
\]
where $\vec n_q$ is the unit normal vector to $\partial Q$ at $q$ directed inward $Q$ if $q\in\partial Q$, with convention $\vec n_{q}=0$ if $q\not \in\partial Q$. This flow preserves the normalized Lebesgue measure $\mu$ on $Q\times S^1$.\\
We assume moreover that every billiard trajectory meets $\partial Q$ (finite horizon assumption).

Let us write $\Pi_Q:Q\times S^1\rightarrow Q$ and $\Pi_V:Q\times S^1\rightarrow S^1$ for the canonical projections given respectively by
$\Pi_Q(q,\vec v)=q$ and $\Pi_V(q,\vec v)=\vec v$.

\begin{theorem}[Visits of the billiard flow to a finite union of shrinking balls for the position]\label{thm:billplusieursboules}
Let $q_1,...,q_J\in Q$ 
be pairwise distinct positions in the billiard domain and $r_1,...,r_j$ be positive real numbers. We set $d_j=2$ if $q_j\not\in\partial Q$
and $d_j=1$ if $q_j\in\partial Q$ and $d=\sum_{j=1}^J d_jr_j$.

Then, the family of processes
$$\sum_{j=1}^J\sum_{t\ :\ (Y_s(y))_s\mbox{ enters } B(q_j,\varepsilon r_j)\times S^1\mbox{ at time t}}\delta_{\left(\frac{d\pi\varepsilon t}{Area(Q)},j,\frac{\Pi_Q(Y_t(y))-q_j}{r_j\eps},\Pi_V(Y_t(y))\right)}$$
converges in distribution (when $y$ is distributed with respect to any probability measure absolutely
continuous with respect to the Lebesgue measure on $\mathcal M$) as $\eps\rightarrow 0$ to 
a Poisson Point Process with intensity $\lambda\times \tilde m_0$
where $\tilde m_0$ is the probability measure on $V:=\{1,...,J\}\times S^1\times S^1$ with density $(j,p,\vec u)\mapsto\frac{r_j}{2d\pi} \langle(-p),\vec u\rangle^+
\mathbf 1_{\{\langle p,\vec n_{q_j}\rangle\ge 0\}}$.\\
\end{theorem}
Observe that if $q_j\in\partial Q$, the set of $p\in S^1$ satisfying $\langle p,\vec n_{q_j}\rangle\ge 0$ is a semicircle, whereas it is the full circle $S^1$ when $q_j$ is in the interior of $Q$.

This result has already been proved in \cite[Theorem 4.4]{ps20} for $J=1$ and Lebesgue-almost every position $q_1$.
The extension to a finite number of points is relatively easy. The most difficult part is to treat all the possibles positions in the billiard domain.

Along the paper we provided various applications of this theorem to different questions. We present here a result on the closest approaches to a given point in the billiard table by the orbit of the billiard flow.
\begin{example}
	Consider the billiard flow associated to a Sinai billiard with finite horizon in a domain $Q\subset \mathbb T^2$. 
	Consider a fixed position $q_0\in Q$. Set $d=2-\mathbf 1_{q_0\in\partial Q}$. 
	During each visit of the flow to $B(q_0,\eps)$, the closest distance to $q_0$ is given by $L_0(q,\vec v):=\eps|\sin\angle(\overrightarrow{qq_0},\vec v)|$ where $(q,\vec v)$ is the entry point.\\
	Then the family of closest approach point process
	\[
	\left(\mathcal \C_\eps:=\mathcal N(Y,B(q_0,\eps)\times S^1,d\eps/Area(Q),\eps^{-1}L_0\right)_{\eps>0}
	\]
	on $[0,+\infty)\times [0,1]$ converges in distribution (with respect to any probability measure absolutely continuous with respect to the Lebesgue measure on $Q\times S^1$) to a PPP with intensity 1.
\end{example}
\begin{proof}
	Due to Theorem~\ref{thm:billplusieursboules}, 
	the family of spatio-temporal processes 
	\[
	(\N_\eps:=\mathcal N(Y,B(q_0,\eps)\times S^1,d\eps/Area(Q),H_\eps)_{\eps>0}
	\]
	with
	$H_\eps(q,\vec v)=(\eps^{-1}\overrightarrow{q_0q},\vec v)$
	converges in distribution (with respect to any probability measure absolutely continuous with respect to Lebesgue on $Q\times S^1$)
	to a PPP of intensity $\lambda\times\tilde m_0$ where $\tilde m_0$ is the probability measure on $S^1\times S^1$ with density $(p,\vec u)\mapsto\frac{1}{2d\pi} \langle(-p),\vec u\rangle^+
	\mathbf 1_{\{\langle p,\vec n_{q_0}\rangle\ge 0\}}$ (where $\vec n_{q_0}$
	is the unit normal vector to $\partial Q$ at $q_0$ directed inward $Q$ if $q_0\in\partial Q$, $\vec n _{q_0}=0$ otherwise).\\
	Observe that
		\[
		\mathcal \C_\eps=\tilde G(\N_\eps),
		\]
		with $\tilde G(t,p,\vec u)=(t,G(p,\vec u))$ where $G(p,\vec u)=(t,|\sin\angle(-p,\vec u)|)$. Thus $(\mathcal \C_\eps)_\eps$ converges strongly in distribution to the PPP
	with intensity $\lambda\times G_*(\tilde m_0)$ and it remains to identify $\tilde m_1=G_*(\tilde m_0)$. Due to the transfer formula, we obtain
	\begin{align*}
		\int_0^\infty \!\!\!\!\! h(G(p,\vec u))\, d\tilde m_0(p,\vec u)&= \frac1{2d\pi}
		\int_{S^1\times S^1}\!\!\!\!\! h(|\sin\angle(-2p,\vec u)|) (\cos\angle(-p,\vec u))^+
		\mathbf 1_{\{\langle p,\vec n_{q_j}\rangle\ge 0\}}\, dp\, d\vec u\\
		&=\frac1{2}
		\int_{-\frac\pi 2}^{\frac\pi 2} h(|\sin\varphi|) \cos\varphi\, d\varphi\\
		&=\int_{0}^{\frac\pi 2} h(\sin\varphi) \cos\varphi\, d\varphi=\int_{0}^{1} h(y)\, dy\, .
	\end{align*}
\end{proof}

\begin{proof}[Proof of Theorem~\ref{thm:billplusieursboules}]
Due to \cite[Theorem 1]{Zweimuller:2007}, it is enough to prove the result for the convergence in distribution with respect to $\mu$.
Assume $\eps>\min_{j\ne j'}\frac{q_jq_{j'}}4$.
We use the representation of the billiard flow as a special flow over the
discrete time billiard system $(M,\nu,F)$ corresponding to collision times and with $\tau$ the length of the free flight before next collision.\\
Set $\widetilde A_\eps=\bigcup_{j=1}^J\widetilde A_\eps^{(j)}$, where
$\widetilde A_\eps^{(j)}$ is the set of the configuration entering in $A^{(j)}_\eps:=(Q\cap B(q_j,\eps))\times S^1$, i.e.
$\widetilde A_\eps^{(j)}$ is the set of $(q,\vec v)\in (Q\cap 
\partial B(q_j,\eps])\times S^1$ s.t. $\langle \vec{qq_0},\vec v\rangle>0$.
Set also $A_\eps:=\bigcup_{j=1}^JA_\eps^{(j)}$.\\
Set $h'_\eps:=d\pi\eps/Area(Q)$ and $H_\eps(q',\vec v)=(j,\frac{\overrightarrow{q_jq'}}{r_j\eps},\vec v)$ if $q'\in\partial B(q_j,r_j\eps)$.
Here $M$ is the set of reflected unit vectors based on $\partial Q$, $\nu$ is the probability measure with density proportional to $(q,\vec v)\mapsto \langle \vec n(q),\vec v\rangle$, where $\vec n(q)$ is the unit vector normal to $\partial Q$ at $q$ directed towards $Q$ and $F:M\rightarrow M$ is the transformation mapping a configuration at a collision time to the configuration corresponding to the next collision time.\\
The normalizing function $G_\eps$ is given by
$G_\varepsilon(x)=H_\eps(Y_{\tau^{(Y)}_{\widetilde A_\varepsilon}(x)}(x))$ with $\tau^{(Y)}_{\widetilde A_\varepsilon}(y):=\inf\{t>0\ :\ Y_t(y)\in \widetilde A_\varepsilon\}$.\\
As in the setting of Theorem~\ref{THMflow}, we write $\Pi$ for the projection on $M$, that is $\Pi(q',\vec v)=(q,\vec v)$ is the post-collisional vector at the previous collision time. We take here $h_\eps:=\nu(\Pi(\widetilde A_\eps))$.\\
As for \cite[Theorem 4.4]{ps20}, we will apply \cite[Proposition 3.2]{ps20} after checking its assumptions. We define $\widetilde A_\varepsilon^{(j)}:=\{(q,\vec v)\in\partial B(q_j,\varepsilon)\times S^1\ :\ \langle \overrightarrow{qq_j},\vec v\rangle\ge 0\}$.
\begin{enumerate}
\item \underline{Measure of the set}. We have to adapt slightly the first item of the proof of \cite[Theorem 4.4]{ps20} which deals with the asymptotic behaviour of $\nu(B_\eps)$ with $B_\eps:=\Pi(\widetilde A_\eps)$. Observe that $B_\eps=\bigcup_{j=1}^JB_\eps^{(j)}$ with $B_\eps^{(j)}:=\Pi(\widetilde A_\eps^{(j)})$,  i.e. $B^{(j)}_\eps$ is the set of configurations 
$(q,\vec v)\in M$ such that the billiard trajectory $(Y_t(q))_{t\ge 0}$
will enter $B(q_j,\eps r_j)$ before touching $\partial Q$.
As seen in \cite[Lemma 5.1]{ps10},
\[
\mbox{if }q_j\in Q\setminus\partial Q,\quad \nu(B^{(j)}_\eps)=\frac{|Q\cap \partial B(q_j,r_j\eps) |}{|\partial Q|}=\frac{2\pi r_j\varepsilon}{|\partial Q|}.
\]
With exactly the same proof, we obtain that 
\[
\mbox{if }q_j\in\partial Q, \quad \nu(B^{(j)}_\eps)=\frac{|Q\cap \partial B(q_j,r_j\eps) |}{|\partial Q|}\sim\frac{\pi r_j\varepsilon}{|\partial Q|}\, .
\]
Moreover, for every distinct $j,j'$, 
$B^{(j)}_\eps\cap B^{(j')}_\eps$ is contained in $\Pi (B(x_{j,j'}, K_{j,j'}\eps)\cup B(x_{j',j},K_{j,j'}\eps))$
where $x_{j,j'}=\left(q_j,\overrightarrow{q_jq'_j}{q_jq'_j}\right)$ and $K_{j,j'}=\max\left(1,\frac{3}{q_jq_{j'}}\right)$.
So, due to \cite[Lemma 5.1]{ps10}, $\nu(B^{(j)}_\eps\cap B^{(j')}_\eps )=\mathcal O(\eps^2)=o(\eps)$. Hence we conclude that
\[\nu(B_\eps)\sim \sum_{j=1}^J\nu(B_\eps^{(j)})\sim\frac{d\pi \eps}{|\partial Q|}\, ,\]
as $\eps\rightarrow 0$.
\item Observe that
\[
\mathcal N(Y, A_\eps,h'_\eps,H_\eps)=\sum_{j=1}^J\mathcal N(Y, A^{(j)}_\eps,h'_\eps,H_\eps) \ge\mathcal N(Y, A'_\eps,h'_\eps,H_\eps) \, ,
\]
where $A'_\eps=\bigcup_{j=1}^J\Pi^{-1}(\Pi(\widetilde A^{(j)}_\eps))\setminus\bigcup_{j'\ne j}\Pi^{-1}(\Pi(\widetilde A^{(j')}_\eps))$
and that, for all $T>0$,
\begin{align*}
&\mathbb E_{\mu}\left[\left(\mathcal N(Y, A_\eps,h'_\eps,H_\eps) -\mathcal N(Y, A'_\eps,h'_\eps,G_\eps) \right)([0,T]\times V)\right]\\
&\le   \frac {T\max\tau}{2h_\eps (\min\tau)^2}\sum_{j,j'\, :\, j\ne j'}\nu\left(\widetilde A_\eps^{(j)}\cap\widetilde A_{\eps}^{(j')}\right)=o(1)\, ,
\end{align*}
where we used the representation of $Y$ as a special flow over $(M,\nu,F)$
due to the fact, proved in the previous item, that for any distinct labels $j,j'$, $\nu\left(\widetilde A_\eps^{(j)}\cap\widetilde A_{\eps}^{(j')}\right)=o(\eps)$.
Thus it is enough to prove the convergence in distribution of
$\mathcal N(Y, A'_\eps,h'_\eps,H_\eps)$ with respect to $\mu$.
\item The same argument ensures that, with respect to $\nu$, the convergence in distribution of
$\mathcal N(F,B_\eps,h_\eps,G_\eps)$ to $\mathcal P$ is equivalent to the convergence in distribution of $\mathcal N(F,B'_\eps,h_\eps,G_\eps)$, with $B'_\eps:=\Pi(A'_\eps)$.
\item 
Note that $\nu((\partial B_\varepsilon)^{[\varepsilon^\delta]})
=o(\nu(B_\varepsilon))$, for every $\delta>1$.
\item Due to Lemma~\ref{quickreturnbilliard}, for every $\sigma>1$, $  \nu(\tau_{B_\varepsilon}\le\eps^{-\sigma}|B_\varepsilon)=o(1)$, where $\tau_B$ is here the first time $k\ge 1$ at which $F^k(\cdot)\in B$.
\item Now let us prove that $(\nu(G_\eps^{-1}(\cdot)|B_\eps))_{\eps>0}$
converges to $\tilde m_0$ as $\eps\rightarrow 0$.\\
Let us consider the measure $\tilde\mu$ on $\{1,...,J\}\times S^1\times S^1$ with density $(j,p,\vec u)\mapsto r_j \langle(-p),\vec u\rangle^+$.\\
Observe first that $\tilde m_0=\tilde\mu(\cdot|A)$ with $A:=\bigcup_{j=1}^JA^{(j)}$ and
\[
A^{(j)}:=\left\{(p,\vec u)\in S^1\times S^1\, :\, \langle (-p),\vec u\rangle\ge 0,\ \langle p,\vec n_{q_j}\rangle\ge 0\right\}\, 
\]
and second that $\nu(G_\eps^{-1}(\cdot)|B_\eps)=\widetilde\mu(\cdot|G_\eps(B_\eps))$. But
\begin{align*}
\tilde\mu\left(A\setminus G_\eps(B_\eps)\right) &\le \sum_{j=1}^{J}\tilde\mu\left( H_\eps\left(Y_{\tau_{\widetilde A_\eps^{(j)}}^{(Y)}(\cdot)}\left(\bigcup_{j'\ne j}(B_\eps^{(j)} \cap B_\eps^{(j')})\right)\right)\right)\\
&\le \sum_{j=1}^J 2\max\tau|\partial Q|r_j\eps\nu\left(
\bigcup_{j\ne j'}(B_\eps^{(j)} \cap B_\eps^{(j')})\right)=o(\nu(B_\eps))
\end{align*}
and $G_\eps(B_\eps)\setminus A$ corresponds to points $(p,\vec u)\in S^1\times S^1$ with $q_j\in\partial Q$ with $0<\langle p,\vec u\rangle\le O(\eps)$, thus
\[
\tilde\mu\left(G_\eps(B_\eps)\setminus A\right) =O(\eps)\, .
\]
This ends the proof of the convergence in distribution of
the family of measures $(\nu(G_\varepsilon^{-1}(\cdot)|B_\varepsilon))_{\eps>0}$
to $\tilde m_0$ as $\eps\rightarrow 0$.
\item For the construction of $\mathcal W$ we use \cite[Proposition 3.4]{ps20}.
\end{enumerate}
Thus, due to  \cite[Proposition 3.2]{ps20}, we conclude the convergence of distribution with respect to $\nu$ of
$(\mathcal N(F,B_\eps,h_\varepsilon,G_\varepsilon))_{\varepsilon>0}$ and so, due to (ii), of
$(\mathcal N(F,B'_\eps,h_\varepsilon,G_\varepsilon))_{\varepsilon>0}$
to a PPP $\mathcal P$ with intensity $\lambda\times\tilde m_0$.
Applying now Theorem~\ref{THMflow}, we deduce the strong convergence in distribution of $(\mathcal N(F,A'_\eps,h_\varepsilon/\mathbb E_\nu[\tau],H_\varepsilon))_{\varepsilon>0}$
to $\mathcal P$ and so, due to (iii), the convergence in distribution with respect to $\mu$ of $(\mathcal N(F,A_\eps,h_\varepsilon/\mathbb E_\nu[\tau],H_\varepsilon))_{\varepsilon>0}$ to
$\mathcal P$. Now we conclude by \cite[Theorem 1]{Zweimuller:2007} and
by noticing that 
\[
\frac{h_\eps}{\mathbb E_\nu[\tau]}=\frac{d\pi \eps}{|\partial Q|\mathbb E_\nu[\tau]}=\frac{d\pi \eps}{Area(Q)}=h'_\eps\, .
\]
\end{proof}
\begin{lemma}\label{quickreturnbilliard}
\begin{equation}\label{quickreturn}
\forall \sigma\in(0,1),\quad \nu(\tau_{B_\eps}\le \eps^{-\sigma}|B_\eps)= o(1)
\end{equation}
\end{lemma}
\begin{proof}
This point corresponds to the the second item of the proof  of \cite[Theorem 4.4]{ps20}, which for Lebesgue-almost every point came from \cite[Lemma 6.4]{ps10}.
To prove \eqref{quickreturn}, we write
\begin{equation}\label{decompquickreturn}
\nu(\tau_{B_\eps}\le \eps^{-\sigma}|B_\eps)\le\sum_{k=1}^{\lfloor \eps^{-\sigma}\rfloor}\nu(F^{-n}(B_\eps)|B_\eps)\, .
\end{equation}
Thus our goal to bound $\nu(F^{-n}(B_\eps)|B_\eps)$.
\medskip

\noindent\underline{Step 1: Useful notations}.\\
We parametrize $M$ by $\bigcup_{i=1}^I
     \{i\}\times (\mathbb R/|\partial O_i|\mathbb Z)\times\left[-\frac \pi 2;\frac \pi 2\right]$.
A reflected vector $(q,\vec v)\in M$ is represented by $(i,r,\varphi)$ if $q\in\partial\Gamma_i$
as curvilinear absciss $r$ $\partial O_i$ and if $\varphi$
is the angular measure in $[-\pi/2,\pi/2]$ of $(\vec{n}(q),\vec{v})$ where
$\vec{n}(q)$ is the normal vector to $\partial Q$ at $q$.\\
For any $C^1$-curve $\gamma$ in $M$, we write $\ell(\gamma)$
for the euclidean length in the $(r,\varphi)$ coordinates of $\gamma$.
If moreover $\gamma$ is given in coordinates by $\varphi=\phi(r)$, then we also write $p(\gamma):=\int_\gamma \cos(\phi(r))\, dr$.
We define the time until the next reflection in the future by
$$\tau
(q,\vec {v}):=\min\{s>0\ :\ q+ s\vec{v}\in \partial Q\}\, .$$
It will be useful to define $\mathcal S_0:=\{\varphi=\pm\pi/2\}$.
Recall that, for every $k\ge 1$, $F^k$ defines a $C^1$-diffeomorphism from $M\setminus \mathcal S_{-k}$ to $M\setminus \mathcal S_k$ with $\mathcal S_{-k}:=
\bigcup_{m=0}^{k}F^{-m}(\mathcal S_0
)$ and
$\mathcal S_{k}:=
\bigcup_{m=0}^{k}F^{m}(\mathcal S_0
)$.
\medskip

\noindent\underline{Step 2: Geometric study of $B_\eps$ and of $F(B_\eps)$}.\\

Moreover the boundary of each connected component of $B_\eps$ (resp. $F(B_\eps)$) is made with a bounded number of $C^1$ curves of the following forms:
\begin{itemize}
\item curves of $\mathcal S_0$,
corresponding, in $(r,\varphi)$-coordinates, to $\{\varphi=\pm\frac\pi 2\}$.
\item $C^1$ curves of $F^{-1}(\mathcal S_0)$ (resp. $F(\mathcal S_0)$), which have the form $\varphi=\phi(r)$
with $\phi$ a $C^1$ decreasing (resp. increasing) function satisfying $\min\kappa\le |\phi'(r)|\le \max \kappa+\frac 1{\min\tau}$,
where $\kappa(q)$ is the curvature of $\partial Q$ at $q\in \partial Q$ and where
$\tau$ is the free flight length before the next collision time.
\item \underline{if $q_0\not\in \partial Q$}: $C^1$ curves, corresponding to
the set 
of points $x=(q,\vec v)\in M$ (resp. $F(x)$) such that 
$[\Pi_Q(x),\Pi_Q(F(x))]$
is tangent to $\partial B(q_0,\eps)$. 
These curves have the form $\varphi=\phi_\varepsilon(r)$
with $\phi_\eps$ a decreasing (resp. increasing) function satisfying $\min\kappa\le |\phi_\eps'(r)|\le \max \kappa+\frac 1{d(q_0,\partial Q)-\eps}
\le \max \kappa+\frac 2{\tau_0}$),
with $\tau_0:=d(q_0,\partial Q)$
as soon as $\eps<\frac{\tau_0}2$.
\item  \underline{if $q_0\in \partial Q$}: $C^1$ curves, corresponding to
the set 
of points $x=(q,\vec v)\in M$ (resp. $F(x)$) such that 
$[\Pi_Q(x),\Pi_Q(F(x))]$
is tangent to $\partial B(q_0,\eps)$ or 
such that $\Pi_Q(F(x))$ is an extremity of $B(q_0,\eps)\cap Q$ and $[\Pi_Q(x),\Pi_Q(F(x))]$ contains no other point of $B(q_0,\eps)$.
These curves have the form $\varphi=\phi_\varepsilon(r)$
with $\phi_\eps$ a decreasing (resp. increasing) function satisfying $\min\kappa\le |\phi_\eps'(r)|
$.\\
The points $x=(q,\vec v)\in M$, with $d(q,q_0)\ll 1$
 quasi-immediately entering (resp. exiting) $B(q_0,\eps)\times S^1$
are contained in a union $R_\eps$ of  two rectangles of width 
$\mathcal O(\eps^{1/2})$ for the position (around $q_0$) and of width $\mathcal O(\eps)$ for the velocity direction (around the tangent vectors to $\partial Q$ at $q_0)$.\\
In $B_\eps\setminus R_\eps$ (resp. $F(B_\eps)\setminus (R_\eps\cup \Pi_Q^{-1}(B(q_0,\eps)))$) we also have $|\phi_\eps'(r)|\le \max \kappa+\frac 2{\tau_0}$
with $\tau_0:=\min\tau$
as soon as $\eps<\frac{\tau_0}2$.
\end{itemize}

We say that a {\bf curve $\gamma$ of $M$ satisfies
assumption (C)} if it is 
given by $\varphi=\phi(r)$ with $\phi$
$C^1$-smooth, increasing and such that $\min\kappa\le \phi'\le \max\kappa+\frac 2{\tau_0}$.
We recall the following facts.
\begin{itemize}
\item There exist $C_0,C_1>0$ and $\lambda_1>1$ such that, for every 
$\gamma$ satisfying Assumption (C) and every integer $m$ such that
$\gamma\cap\mathcal S_{-m}=\emptyset$, $F^m\gamma$ is a $C^1$-smooth
curve satisfying assumption (C) and
$C_1 p(F^m\gamma)\ge \lambda_1^mp(\gamma)$ and $\ell(\gamma)\le C_0
      \sqrt{p(F\gamma)}$.
\item There exist $C_2>0$ and $\lambda_2>\lambda_1^{1/2}$ such that, for every integer $m$,
the number of connected components of $M\setminus \mathcal S_{-m}$ is less than
$C_2\lambda_2^m$. Moreover $\mathcal S_{-m}$ is made of curves $\varphi=\phi(r)$
with $\phi$ $C^1$-smooth and strictly decreasing.
\item If $\gamma\subset M\setminus \mathcal S_{-1}$ is given by $\varphi=\phi(r)$ or $r=\mathfrak r(\varphi)$
with $\phi$ or $\mathfrak r$ increasing and $C^1$ smooth, then $F\gamma$
is $C^1$, is given by $\varphi=\phi_1(r)$ with
$\min\kappa\le \phi_1'\le\max\kappa+\frac 1{\min\tau}$.
Moreover $\int_{F\gamma}\, d\varphi\ge \int_{\gamma}\, d\varphi$.
\end{itemize}
We observe that there exist $
K'_0>0$ and $\varepsilon_0>0$ such that, for every $\varepsilon\in(0,\varepsilon_0)$,
$F(B_\eps)\setminus R_\eps$ is made of a bounded number of connected components $V_\eps^{(i)}$
each of which is a strip of width at most $K'_0\eps$ of the following form in $(r,\varphi)$-coordinates:
\begin{itemize}
\item $\{(r,\varphi)\, :\, r\in J,\ \phi^{(i)}_1(r)\le\varphi\le\phi^{(i)}_2(r)\}$ (with $J$ an interval) and is delimited by two continuous piecewise $C^1$ curves $\gamma_j$ given by $\varphi=\phi_j(r)$ satisfying assumption (C) and $\Vert\phi^{(i)}_1-\phi^{(i)}_2\Vert_\infty\le K'_0\varepsilon$.
\item or possibly, if $q_0\in\partial Q$, $\{(r,\varphi)\, :\,   r^{(i)}_{1,\eps}\le r\le r^{(i)}_{2,\eps}\}$
with $|r^{(i)}_{1,\eps}-r^{(i)}_{2,\eps}|\le K'_0\eps$.
\end{itemize}
In particular, with the previous notations, any connected component $V_\eps^{(i)}$ of
$F(B_\eps)\setminus R_\eps$ has the form $\bigcup_{u\in[0,1]}\widetilde\gamma^{(i)}_u$,
where $\widetilde\gamma^{(i)}_u$ corresponds to the graph $\{\psi^{(i)}(u,r)=(r,u\phi^{(i)}_1(r)+(1-u)\phi^{(i)}_2(r))\, :\, r\in J_i\}$ (or possibly $\{\psi^{(i)}(u,\varphi)=(ur^{(i)}_{1,\eps}+(1-u)r^{(i)}_{2,\eps},\varphi),\ \varphi\in J_i \}$  if $q_0\in\partial Q$). Thus
\begin{align}
\nonumber\forall E\in\mathcal B(M),\quad\nu(&E\cap F(B_\eps\setminus R_\eps))\le \frac {Leb(E\cap F(B_\eps\setminus R_\eps))}{2|Q|}\\
\nonumber&\le\sum_i\frac 1{2|\partial Q|} \int_{J_i\times[0,1]} \mathbf 1_{\psi^{(i)}(u,s)
\in E}\left|\frac{\partial}{\partial u}\psi^{(i)}(u,s)\right|
 \,  dsdu\\
&\le\frac{ K'_0\, \eps}{2|\partial Q|}\sup_{[0,1]}\ell(E\cap \widetilde\gamma_u)\, .\label{desint}
\end{align}
\medskip

\noindent\underline{Step 3: Scarcity of very quick returns}.\\
Let us prove the existence of $K_1>0$ such that, 
\begin{equation}\label{lemmaveryquick} 
\forall s\ge 1,\ \forall\eps<\frac{\tau_0}2,\quad \nu(F^{-s-1}(B_\varepsilon)|B_\varepsilon)\le K_1(\lambda_2/\lambda_1^{\frac 12})^s\varepsilon^{\frac 12}.
\end{equation}
Let $u\in(0,\varepsilon)$.
We define $\gamma$ be a connected component of $\widetilde\gamma_u\cap F(B_\varepsilon)\cap F^{-s}(B_\varepsilon
)$.
The curve $\gamma$ satisfies Assumption (C) or is vertical. 
In any case, any connected component of $F(\gamma)$ satisfies Assumption (C)
and $\ell(\gamma)\le C_0\sqrt{p(F(\gamma))}$ (indeed, if $\gamma$ is vertical, then $\ell(\gamma)\le \frac 1{\min\tau}p(F(\gamma))$. It follows
$$\ell(\gamma)
\le C_0\sqrt{p
(F(\gamma))}\le C_0\sqrt{C_1\lambda_1^{1-s}p(F^s\gamma)}\le C'_0\sqrt{C_1\lambda_1^{1-s}K'_0\varepsilon}$$
using first the fact that $F(\gamma)$ is an increasing curve contained in $M\setminus \mathcal S_{-s}$
and second the fact that
 $F^s\gamma$ is 
is an increasing curve satisfying Condition (C) and contained in $B_\varepsilon
$.
Since
$
F(\widetilde\gamma_u)\setminus\mathcal S_s$ contains at most $C_2\lambda_2^s$
connected components, using \eqref{desint}, we obtain
$$ \nu(F^{-s-1}(B_\varepsilon)\cap B_\varepsilon\setminus R_\eps)=\nu(F^{-s}(B_\varepsilon)\cap F(B_\varepsilon\setminus R_\eps))\le 
\frac{ K'_0\, \eps}{2|\partial Q|}\sup_{[0,1]}
C_2\lambda_2^s
C'_0\sqrt{
C_1}\lambda_1^{\frac s2}\varepsilon^{\frac 12}.$$
We conclude by using the fact that $\nu(B_\eps)=\frac{2\pi\varepsilon}{|\partial Q|}$ and that $\nu(R_\eps)=\mathcal O(\eps^{\frac 32})$.
\medskip

\noindent\underline{Step 4: Scarcity of intermediate quick returns}.\\
We prove now that
for any $a>0$, there exists $s_a>0$ such that
\begin{equation}\label{lemmaquick} 
\sum_{n=-a\log\eps}^{\eps^{-s_a}} \nu(B_\eps\cap F^{-n}B_\eps)=o(\nu(B_\eps)).
\end{equation}
Since $\nu(B_\eps)\approx\eps$ and $\nu(R_\eps)=\mathcal O(\eps^{\frac 32})$,
up to adding the condition $s_a<1/2$, it remains to prove \eqref{lemmaquick} with for 
$\nu(B_\eps\cap F^{-n}B_\eps)$ replaced by $\nu((B_\eps\setminus R_\eps)\cap F^{-n}(B_\eps
))$.\\
If $q_0\in\partial Q$ and if $\widetilde\gamma_u$ is vertical, we replace it in the argument below 
by the connected components of $F(\widetilde\gamma_u)$ and will conclude
by noticing that; for any measurable set $A$, $\ell(\widetilde\gamma_u\cap F^{-1}(A))\le C''_0\ell(F(\widetilde\gamma_u\cap A))$.\\
We denote the $k$th homogeneity strip\footnote{see \cite{ChernovMarkarian} for notations and definitions.}  by $\mathbb{H}_k$ for $k\neq0$ and 
set $\mathbb{H}_0 = \cup_{|k|< k_0} \mathbb{H}_k$ for some fixed $k_0$.
Set $s:=\min(-a\log\theta,1)/3$. 
Let $k_\eps=\eps^{-s}$ and $H^\eps=\cup_{|k|\le k_\eps}\mathbb{H}_k$.
For any $u\in[0,1]$, we set 
$\widetilde\gamma_{k,u}=\widetilde\gamma_u\cap \mathbb{H}_k$.
Each $\widetilde\gamma_{k,u}
$ is a weakly homogeneous unstable curve.

We cut each curve $\widetilde\gamma_{k,u}$ into small pieces $\widetilde\gamma_{k,u,i}$ such that 
each $F^j \widetilde\gamma_{k,u,i}$, $j=0,\ldots,n$ is contained in a homogeneity strip and a connected component of $M\setminus \mathcal{S}_1$. 
For $x\in \widetilde\gamma_{k,u,i}$ we denote by $r_n(x)$ the distance (in $F^n \widetilde\gamma_u$) of $F^n(x)$ to the boundary of $F^n \widetilde\gamma_{k,u,i}$. 

Recall that the 
growth lemma~\cite[Theorem 5.52]{ChernovMarkarian} ensures the existence of
$\theta\in(0,1)$, $c>0$ such that, for any weakly homogeneous unstable curve $\gamma$ one has
\begin{equation}\label{GL}
\ell(\gamma\cap \{r_n<\delta\}) \le c\theta^n \delta + c\delta \ell(\gamma)\, .
\end{equation}
Therefore, 
\[
\begin{split}
\ell&(\widetilde\gamma_u
\cap F^{-n} ( B_\eps
)\setminus \mathbb{H}_\eps)
\\
\le 
&\sum_{|k|\le k_\eps} \ell(\cap\{r_n\ge \eps^{1-s}\}\cap F^{-n}(B_\eps
)) + \ell(\widetilde\gamma_{u,k}\cap\{r_n<\eps^{1-s}\}).
\end{split}
\]
The first term inside the above sum is bounded by the sum $\sum_i \ell(\widetilde\gamma_{u,k,i}\cap F^{-n}(B_\eps
)) $ over those $i$'s such that $F^n(\widetilde\gamma_{u,k,i})$ is of size larger than $\eps^{1-s}$. In particular $\ell(\widetilde\gamma_{u,k,i})\ge \eps^{1-s}$. On the other hand, by transversality 
\[
\ell(F^n (\widetilde\gamma_{u,k,i})\cap B_\eps
)\le c\eps.
\]
By distortion (See Lemma 5.27 in \cite{ChernovMarkarian}) we obtain
\[
\ell(\widetilde\gamma_{u,k,i}\cap F^{-n}(B_\eps
)) \le c \eps^{s} \ell(\widetilde\gamma_{u,k,i}).
\]
Summing up over these $i$ gives the first term inside the sum is bounded by
\[
\ell(\widetilde\gamma_{u,k}\cap\{r_n\ge \eps^{1-s}\}\cap F^{-n}(B_\eps)) 
\le c \eps^{s} \ell(\widetilde\gamma_{u,k,i}).
\]
Thus
\[
\ell(\widetilde\gamma_{u,k}\cap\{r_n<\eps^{1-s}\}) \le c \theta^n \eps^{1-s} + c\eps^{1-s} \ell(\widetilde\gamma_{u,k}).
\]
A final summation over $k$ gives
\[
\ell(\widetilde\gamma_{u} \cap F^{-n} (B_\eps)\setminus \mathbb{H}_\eps)
\le c(\eps^{s}+\eps^{1-s})\ell(\widetilde\gamma_{u}) + ck_\eps\theta^n \eps^{1-s}.
\]
This combined with \eqref{desint} leads to
\[
\nu(F(B_\eps\setminus R_\eps)\cap F^{-n}(B_\eps))\le \nu(F(B_\eps\setminus R_\eps)\cap \mathbb H_\eps)+
\mathcal O(\eps^{1+s})=\mathcal O(\eps^s\nu(B_\eps)).
\]
where we use the fact that $B_\eps\setminus \mathbb{H}_\eps$ is contained in a uniformly bounded union of rectangles of horizontal width $\mathcal O(\eps)$ and contained in the $k_\eps^{-2}=\eps^{2s}$-neighbourhood of $\mathcal S_0$. We take $s_a<\min( s,\frac 12)$. 
\medskip

\noindent\underline{Step 5: End of the proof of \eqref{quickreturn}}.\\
Choose $a=1/(4\log(\lambda_2/\lambda_1^{1/2})$.
Observe that, due to \eqref{lemmaveryquick}, we have
$$\sum_{s=1}^{-a\log\eps}\mu(F^{-s}A_\eps|A_\eps)\le
\frac{K_1}{\lambda_2/\lambda_1^{\frac 12}-1}(\lambda_2/\lambda_1^{\frac 12})^{-a\log\eps}\varepsilon^{1/2} \le 
\frac{K_1}{\lambda_2/\lambda_1^{\frac 12}-1}\varepsilon^{1/4} .$$
This combined with \eqref{lemmaquick} leads to
\begin{equation}\label{petitsretours}
\sum_{n=1}^{\varepsilon^{-s_a}}\nu(F^{-n}B_\eps|B_\eps)=o(1)\, .
\end{equation}
Let $\sigma>1$. In view of \eqref{decompquickreturn}, it remains to control
$\nu(F^{-n}B_\eps|B_\eps)$ for the intermediate integers $n$ such that $\varepsilon^{-s_a}\le n\le \varepsilon^{-\sigma}$. We approximate the set $B_\eps$ by the union $\widetilde B_\eps$ of connected components of $M\setminus (\mathcal S_{-k(\eps)}\cup\mathcal S_{k(\eps)})$ that intersects $B_\eps$, with $k(\eps)=\lfloor |\log \eps|^2\rfloor$. There exists $\widetilde C>0$ and $\widetilde\theta\in(0,1)$ such that,
for all positive integer $k$,
the diameter of each connected component of $M\setminus (\mathcal S_{-k}\cup\mathcal S_{k})$ is less than $\widetilde C \widetilde \theta^k$.\\
Thus $B_\eps\subset \widetilde B_\eps$ and $\nu(\widetilde B_\eps\setminus B_\eps)\le \nu\left((\partial B_\eps)^{[\widetilde C\widetilde\theta^{k(\eps)}]}\right)=\mathcal O(\eps \widetilde\theta^{k(\eps)})$. But, due to \cite[Lemma 4.1]{ps15}, we also have
\[
\forall m> 1,\ \forall n\ge 2k(\eps),\quad \nu\left(\widetilde B_\eps\cap F^{-n}\widetilde B_\eps\right)=\nu(\widetilde B_\eps)^2+\mathcal O(n^{-m}\nu(\widetilde B_\eps))\, .
\]
Since $k(\eps)=o(\eps^{-s_a})$ and thus
\[
\forall m>1,\quad
\sum_{n=\varepsilon^{-s_a}}^{\varepsilon^{-\sigma}}\nu(F^{-n}B_\eps|B_\eps)
\le \mathcal O\left( \eps^{1-\sigma}+\eps^{s_a(m-1)-\sigma}+\widetilde\theta^{k(\eps)}\right)
=o(1)\, ,
\]
as $\eps\rightarrow 0$,
since $\sigma<1$, $\widetilde\theta\in(0,1)$, $k(\eps)\rightarrow+\infty$ and by taking $m>1+\frac\sigma{s_a}$.
This combined with \eqref{petitsretours} and \eqref{decompquickreturn} ends the proof of  \eqref{quickreturn}.
\end{proof}

%
%
%
%
%
%


\end{appendix}

\end{document}